\documentclass[12pt]{article}
\usepackage{amsfonts}
\usepackage{amsmath}
\usepackage{amssymb}
\usepackage{amsthm}
\usepackage{enumerate}
\makeatletter
\def\imod#1{\allowbreak\mkern10mu({\operator@font mod}\,\,#1)}
\makeatother

\newtheorem{theorem}{Theorem}[section]

\newtheorem{fact}{Fact}[section]
\newtheorem{conjecture}{Conjecture}[section]

\theoremstyle{definition}
\newtheorem{definition}{Definition}[section]
\newtheorem{remark}{Remark}[section]

\begin{document}
\begin{center}
\vskip 1cm{\LARGE\bf On 2-powerfully Perfect Numbers in Three Quadratic Rings
\vskip 1cm
\large
Colin Defant\footnote{This work was supported by National Science Foundation grant no. 1262930.}\\
Department of Mathematics\\
University of Florida\\
United States\\
cdefant@ufl.edu}
\end{center}
\vskip .2 in

\begin{abstract}
Using an extension of the abundancy index to imaginary quadratic rings with unique factorization, we define what we call $n$-powerfully perfect numbers in these rings. This definition serves to extend the concept of perfect numbers that have been defined and studied in the integers. We investigate the properties of $2$-powerfully perfect numbers in the rings $\mathcal O_{\mathbb{Q}(\sqrt{-1})}$, $\mathcal O_{\mathbb{Q}(\sqrt{-2})}$, and $\mathcal O_{\mathbb{Q}(\sqrt{-7})}$, the three imaginary quadratic rings with unique factorization in which $2$ is not a prime.  
\end{abstract}

\section{Introduction} 

Throughout this paper, we will let $\mathbb{N}$ denote the set of positive integers, and we will let $\mathbb{N}_0$ denote the set of nonnegative integers.
\par 
The arithmetic functions $\sigma_k$ are defined, for every integer $k$, by \\ $\displaystyle{\sigma_k(n)=\sum_{\substack{c\vert n\\c>0}}c^k}$. For each integer $k\neq 0$, $\sigma_k$ is multiplicative and satisfies \\ $\displaystyle{\sigma_k (p^\alpha)=\frac{p^{k(\alpha+1)}-1}{p^k-1}}$ for all (integer) primes $p$ and positive integers $\alpha$. The abundancy index of a positive integer $n$ is defined by $\displaystyle{I(n)=\frac{\sigma_1(n)}{n}}$. A positive integer $n$ is said to be $t$-perfect if $I(n)=t$ for a positive integer $t\geq 2$, and $2$-perfect numbers are called perfect numbers. 
\par 
For any square-free integer $d$, let $\mathcal O_{\mathbb{Q}(\sqrt{d})}$ be the quadratic integer ring given by \[\mathcal O_{\mathbb{Q}(\sqrt{d})}=\begin{cases} \mathbb{Z}[\frac{1+\sqrt{d}}{2}], & \mbox{if } d\equiv 1\imod{4}; \\ \mathbb{Z}[\sqrt{d}], & \mbox{if } d\equiv 2, 3 \imod{4}. \end{cases}\] 
\par 
Throughout the remainder of this paper, we will work in the rings $\mathcal O_{\mathbb{Q}(\sqrt{d})}$ for different specific or arbitrary values of $d$. We will use the symbol ``$\vert$" to mean ``divides" in the ring $\mathcal O_{\mathbb{Q}(\sqrt{d})}$ in which we are working. 
Whenever we are working in a ring other than $\mathbb{Z}$, we will make sure to emphasize when we wish to state that one integer divides another in $\mathbb{Z}$. 
For example, if we are working in $\mathbb{Z}[i]$, the ring of Gaussian integers, we might say that $1+i\vert 1+3i$ and that $2\vert 6$ in $\mathbb{Z}$. We will also refer to primes in $\mathcal O_{\mathbb{Q}(\sqrt{d})}$ as ``primes," whereas we will refer to (positive) primes in $\mathbb{Z}$ as ``integer primes." For an integer prime $p$ and a nonzero integer $n$, we will let $\upsilon_p(n)$ denote the largest integer $k$ such that $p^k\vert n$ in $\mathbb{Z}$. For a prime $\pi$ and a nonzero number $x\!\in\!\mathcal O_{\mathbb{Q}(\sqrt{d})}$, we will let $\rho_\pi(x)$ denote the largest integer $k$ such that $\pi^k\vert x$.  
Furthermore, we will henceforth focus exclusively on values of $d$ for which $\mathcal O_{\mathbb{Q}(\sqrt{d})}$ is a unique factorization domain and $d<0$. In other words, $d\in K$, where we will define $K$ to be the set $\{-163,-67,-43,-19,-11,-7,-3,-2,-1\}$. The set $K$ is known to be the complete set of negative values of $d$ for which $\mathcal O_{\mathbb{Q}(\sqrt{d})}$ is a unique factorization domain \cite{Stark67}.  
\par 
For an element $a+b\sqrt{d}\in\mathcal O_{\mathbb{Q}(\sqrt{d})}$ with $a,b\in \mathbb{Q}$, we define the conjugate by $\overline{a+b\sqrt{d}}=a-b\sqrt{d}$. The norm and absolute value of an element $z$ are defined, respectively, by $N(z)=z\overline{z}$ and $\vert z\vert=\sqrt{N(z)}$. We assume familiarity with the properties of these object, which are treated in Keith Conrad's online notes \cite{Conrad}. For $x,y\in\mathcal O_{\mathbb{Q}(\sqrt{d})}$, we say that $x$ and $y$ are associated, denoted $x\sim y$, if and only if $x=uy$ for some unit $u$ in the ring $\mathcal O_{\mathbb{Q}(\sqrt{d})}$. Furthermore, we will make repeated use of the following well-known facts. 
\begin{fact} \label{Fact1.1}
Let $d\!\in\! K$. If $p$ is an integer prime, then exactly one of the following is true. 
\begin{itemize}
\item $p$ is also a prime in $\mathcal O_{\mathbb{Q}(\sqrt{d})}$. In this case, we say that $p$ is inert in $\mathcal O_{\mathbb{Q}(\sqrt{d})}$. 
\item $p\sim \pi^2$ and $\pi\sim\overline{\pi}$ for some prime $\pi\in \mathcal O_{\mathbb{Q}(\sqrt{d})}$. In this case, we say $p$ ramifies (or $p$ is ramified) in $\mathcal O_{\mathbb{Q}(\sqrt{d})}$. 
\item $p=\pi\overline{\pi}$ and $\pi\not\sim\overline{\pi}$ for some prime $\pi\in\mathcal O_{\mathbb{Q}(\sqrt{d})}$. In this case, we say $p$ splits (or $p$ is split) in $\mathcal O_{\mathbb{Q}(\sqrt{d})}$.
\end{itemize}
\end{fact}
\begin{fact} \label{Fact1.2}
Let $d\!\in\! K$. If $\pi\!\in\!\mathcal O_{\mathbb{Q}(\sqrt{d})}$ is a prime, then exactly one of the following is true. 
\begin{itemize}
\item $\pi\sim q$ and $N(\pi)=q^2$ for some inert integer prime $q$. 
\item $\pi\sim\overline{\pi}$ and $N(\pi)=p$ for some ramified integer prime $p$. 
\item $\pi\not\sim\overline{\pi}$ and $N(\pi)=N(\overline{\pi})=p$ for some split integer prime $p$. 
\end{itemize}
\end{fact}
\begin{fact} \label{Fact1.3}
If $d\!\in K\!$, $q$ is an integer prime that is inert in $\mathcal O_{\mathbb{Q}(\sqrt{d})}$, and $x\in\mathcal O_{\mathbb{Q}(\sqrt{d})}\backslash\{0\}$, then $\upsilon_q(N(x))$ is even and $\rho_q(x)=\frac{1}{2}\upsilon_q(N(x))$.
\end{fact}
\begin{fact} \label{Fact1.4}
Let $p$ be an odd integer prime. Then $p$ ramifies in $\mathcal O_{\mathbb{Q}(\sqrt{d})}$ if and only if $p\vert d$ in $\mathbb{Z}$. If $p\nmid d$ in $\mathbb{Z}$, then $p$ splits in $\mathcal O_{\mathbb{Q}(\sqrt{d})}$ if and only if $d$ is a quadratic residue modulo $p$. Note that this implies that $p$ is inert in $\mathcal O_{\mathbb{Q}(\sqrt{d})}$ if and only if $p\nmid d$ in $\mathbb{Z}$ and $d$ is a quadratic nonresidue modulo $p$. 
Also, the integer prime $2$ ramifies in $\mathcal O_{\mathbb{Q}(\sqrt{-1})}$ and $\mathcal O_{\mathbb{Q}(\sqrt{-2})}$, splits in $\mathcal O_{\mathbb{Q}(\sqrt{-7})}$, and is inert in $\mathcal O_{\mathbb{Q}(\sqrt{d})}$ for all $d\in K\backslash\{-1,-2,-7\}$.
\end{fact}
\begin{fact} \label{Fact1.5}
Let $\mathcal O_{\mathbb{Q}(\sqrt{d})}^*$ be the set of units in the ring $\mathcal O_{\mathbb{Q}(\sqrt{d})}$. Then $\mathcal O_{\mathbb{Q}(\sqrt{-1})}^*=\{\pm 1,\pm i\}$, $\displaystyle{\mathcal O_{\mathbb{Q}(\sqrt{-3})}^*=\left\{\pm 1,\pm \frac{1+\sqrt{-3}}{2},\pm \frac{1-\sqrt{-3}}{2}\right\}}$, and $\mathcal O_{\mathbb{Q}(\sqrt{d})}^*=\{\pm 1\}$ \\ 
whenever $d\in K\backslash\{-1,-3\}$. 
\end{fact}
\par 
For a nonzero complex number $z$, let $\arg (z)$ denote the argument, or angle, of $z$. We convene to write $\arg (z)\in [0,2\pi)$ for all $z\in\mathbb{C}$. For each $d\in K$, we define the set $A(d)$ by 
\[A(d)=\begin{cases} \{z\in\mathcal O_{\mathbb{Q}(\sqrt{d})} \backslash\{0\}: 0\leq \arg (z)<\frac{\pi}
{2}\}, & \mbox{if } d=-1; \\ \{z\in\mathcal O_{\mathbb{Q}(\sqrt{d})} \backslash\{0\}: 0\leq \arg (z)<\frac{\pi}
{3}\}, & \mbox{if } d=-3; \\ \{z\in\mathcal O_{\mathbb{Q}(\sqrt{d})} \backslash\{0\}: 0\leq \arg (z)<\pi\}, & \mbox{otherwise}. \end{cases}\] 
Thus, every nonzero element of $\mathcal O_{\mathbb{Q}(\sqrt{d})}$ can be written uniquely as a unit times a product of primes in $A(d)$. Also, every $z\in\mathcal O_{\mathbb{Q}(\sqrt{d})}\backslash\{0\}$ is associated to a unique element of $A(d)$. The author has defined analogues of the arithmetic functions $\sigma_k$ in quadratic rings $\mathcal O_{\mathbb{Q}(\sqrt{d})}$ with $d\in K$ \cite{Defant14A}, and we will state the important definitions and properties for the sake of completeness.   
\begin{definition} \label{Def1.1}
Let $d\in K$, and let $n\in \mathbb{Z}$. 
Define the function 
\newline $\delta_n\colon\mathcal O_{\mathbb{Q}(\sqrt{d})}\backslash\{0\}\rightarrow [1,\infty)$ by 
\[\delta_n (z)=\sum_{\substack{x\vert z\\x\in A(d)}}\vert x \vert^n.\]
\end{definition}
\begin{remark} \label{Rem1.1}
We note that, for each $x$ in the summation in the above definition, we may cavalierly replace $x$ with one of its associates. This is because associated numbers have the same absolute value. In other words, the only reason for the criterion $x\!\in\! A (d)$ in the summation that appears in Definition \ref{Def1.1} is to forbid us from counting associated divisors as distinct terms in the summation, but we may choose to use any of the associated divisors as long as we only choose one. This should not be confused with how we count conjugate divisors (we treat $2+i$ and $2-i$ as distinct divisors of $5$ in $\mathbb{Z}[i]$ because $2+i\not\sim 2-i$).  
\end{remark}

\begin{remark} \label{Rem1.2}
We mention that the function $\delta_n$ is different in each ring $\mathcal O_{\mathbb{Q}(\sqrt{d})}$. Perhaps it would be more precise to write $\delta_n(z,d)$, but we will omit the latter component for convenience. We note that we will also use this convention with functions such as $I_n$ (which we will define soon). 
\end{remark}
\par 
We will say that a function $f\colon\mathcal O_{\mathbb{Q}(\sqrt{d})}\backslash\{0\}\!\rightarrow\!\mathbb{R}$ is multiplicative if $f(xy)=f(x)f(y)$ whenever $x$ and $y$ are relatively prime (have no nonunit common divisors). The author has shown that,
for any integer $n$, $\delta_n$ is multiplicative \cite{Defant14A}.
\begin{definition} \label{Def1.2}
For each positive integer $n$, define the function \\ 
$I_n\colon\mathcal O_{\mathbb{Q}(\sqrt{d})}\backslash\{0\}\rightarrow[1,\infty)$ by $\displaystyle{I_n(z)=\frac{\delta_n(z)}{\vert z\vert ^n}}$. 
For a positive integer $t\geq 2$, we say that a number $z\!\in\!\mathcal O_{\mathbb{Q}(\sqrt{d})}\backslash\{0\}$ is \textit{$n$-powerfully $t$-perfect in $\mathcal O_{\mathbb{Q}(\sqrt{d})}$} if $I_n(z)=t$, and, if $t=2$, we simply say that $z$ is \textit{$n$-powerfully perfect in $\mathcal O_{\mathbb{Q}(\sqrt{d})}$}. 
\end{definition} 
As an example, we will let $d=-1$ so that $\mathcal O_{\mathbb{Q}(\sqrt{d})}=\mathbb{Z}[i]$. Let us compute $I_2(9+3i)$. We have $9+3i=3(1+i)(2-i)$, so $\delta_2(9+3i)=N(1)+N(3)+N(1+i)+N(2-i)+N(3(1+i))+N(3(2-i))+N((1+i)(2-i))+N(3(1+i)(2-i))=1+9+2+5+18+45+10+90=180$. Then $\displaystyle{I_2(9+3i)=\frac{180}{N(3(1+i)(2-i))}=2}$, so $9+3i$ is $2$-powerfully perfect in $\mathcal O_{\mathbb{Q}(\sqrt{-1})}$. 
\begin{theorem} \label{Thm1.1}
Let $n\!\in\!\mathbb{N}$, $d\!\in\! K$, and $z_1, z_2, \pi\in\mathcal O_{\mathbb{Q}(\sqrt{d})}\backslash\{0\}$ with $\pi$ a prime. Then, if we are working in the ring $\mathcal O_{\mathbb{Q}(\sqrt{d})}$, the following statements are true. 
\begin{enumerate}[(a)] 
\item The range of $I_n$ is a subset of the interval $[1,\infty)$, and $I_n(z_1)=1$ if and only if $z_1$ is a unit in $\mathcal O_{\mathbb{Q}(\sqrt{d})}$. If $n$ is even, then $I_n(z_1)\in\mathbb{Q}$. 
\item $I_n$ is multiplicative.  
\item $I_n(z_1)=\delta_{-n}(z_1)$. 
\item If $z_1\vert z_2$, then $I_n(z_1)\leq I_n(z_2)$, with equality if and only if $z_1\sim z_2$. 
\end{enumerate}  	
\end{theorem} 
We refer the reader to \cite{Defant14A} for a proof of Theorem \ref{Thm1.1}. The author has already investigated $1$-powerfully $t$-perfect numbers in imaginary quadratic rings with unique factorization, and he has shown that, for any integers $n\geq 3$ and $t\geq 2$, no $n$-powerfully $t$-perfect numbers exist in these rings \cite{Defant14B}. Hence, the remainder of this paper will focus on the interesting topic of $2$-powerfully $t$-perfect numbers. 

\section{Investigating $2$-powerfully $t$-perfect Numbers}
Trying to find $2$-powerfully $t$-perfect numbers is quite a pleasant activity. One reason for this is that $2$ is the only positive integer $n$ for which there exist $n$-powerfully $t$-perfect numbers that are not associated to integers \cite{Defant14B}. For example, in $\mathcal O_{\mathbb{Q}(\sqrt{-1})}$, $3+9i$ is $2$-powerfully perfect, and $30+30i$ is $2$-powerfully $3$-perfect. We will also utilize the helpful that, for any $d\in K$ and $z\in\mathcal O_{\mathbb{Q}(\sqrt{d})}\backslash\{0\}$, we have $N(z)$, $\delta_2(z)\in\mathbb{N}$. In this section, we will focus on the rings $\mathcal O_{\mathbb{Q}(\sqrt{-1})}$, $\mathcal O_{\mathbb{Q}(\sqrt{-2})}$, and $\mathcal O_{\mathbb{Q}(\sqrt{-7})}$, which are the only rings $\mathcal O_{\mathbb{Q}(\sqrt{d})}$ with $d\in K$ in which $2$ is not inert. 
\begin{theorem} \label{Thm2.1} 
Let us work in a ring $\mathcal O_{\mathbb{Q}(\sqrt{d})}$ with $d\in\{-1,-2\}$. Then $2$ ramifies in $\mathcal O_{\mathbb{Q}(\sqrt{d})}$, so we may write $2\sim\xi^2$ for some prime $\xi$ satisfying $\xi\sim\overline{\xi}$ and $N(\xi)=2$. Suppose $z$ is $2$-powerfully perfect in $\mathcal O_{\mathbb{Q}(\sqrt{d})}$ and $\xi\vert z$. Then we may write $z=\xi^\gamma x$, where $\gamma\in\mathbb{N}$, $x\in\mathcal O_{\mathbb{Q}(\sqrt{d})}$, $\xi\nmid x$, and $2^{\gamma+1}-1$ is a Mersenne prime that is inert in $\mathcal O_{\mathbb{Q}(\sqrt{d})}$. Furthermore, there exists an odd positive integer $m$ such that $\delta_2(x)=2^{\gamma+1}m$ and $N(x)=(2^{\gamma+1}-1)m$. 
\end{theorem}
\begin{proof} 
We know the first part of the theorem, which is stated simply to introduce notation. All that we need to prove is the final sentence of the theorem, as well as the fact that $2^{\gamma+1}-1$ is 
a Mersenne prime that is inert in $\mathcal O_{\mathbb{Q}(\sqrt{d})}$. As $z$ is $2$-powerfully perfect in $\mathcal O_{\mathbb{Q}(\sqrt{d})}$, we have 
\[\delta_2(z)=2N(z)=2N(\xi^\gamma)N(x)=2^{\gamma+1}N(x).\]
However, we 
also have 
\[\delta_2(z)=\delta_2(\xi^\gamma)\delta_2(x)=\left(\sum_{j=0}^\gamma N(\xi^j)\right)\delta_2(x)\]
\[=\left(\sum_{j=0}^\gamma 2^j\right)\delta_2(x)=(2^{\gamma+1}-1)\delta_2(x).\] 
Therefore, $2^{\gamma+1}N(x)=(2^{\gamma+1}-1)\delta_2(x)$. As $2^{\gamma+1}-1$ is odd, we find that $2^{\gamma+1}\vert\delta_2(x)$ in $\mathbb{Z}$. We may then write $\delta_2(x)=2^{\gamma+1}m$ for some positive integer $m$. Substituting this new expression for $\delta_2(x)$ into the equation $2^{\gamma+1}N(x)=(2^{\gamma+1}-1)\delta_2(x)$, we find $N(x)=(2^{\gamma+1}-1)m$. This tells us that $m$ is odd because $\xi\nmid x$ (implying that $2\nmid N(x)$ in $\mathbb{Z})$. Suppose that $2^{\gamma+1}-1$ is not a prime in $\mathcal O_{\mathbb{Q}(\sqrt{d})}$ so that we may write $2^{\gamma+1}-1=y_1y_2$, where $y_1, y_2\in\mathcal O_{\mathbb{Q}(\sqrt{d})}$ satisfy $1<N(y_1)\leq N(y_2)<N(2^{\gamma+1}-1)=(2^{\gamma+1}-1)^2$. Then, because $N(y_1)N(y_2)=N(2^{\gamma+1}-1)=(2^{\gamma+1}-1)^2$, we see that $N(y_1)\leq 2^{\gamma+1}-1$. Now, let $\pi_0$ be a prime that divides $y_1$. Then $\pi_0\vert N(x)$, which implies that either $\pi_0\vert x$ or $\overline{\pi_0}\vert x$. If $\pi_0\vert x$, write $\pi=\pi_0$. Otherwise, write $\pi=\overline{\pi_0}$.   
Then $N(\pi)\leq N(y_1)\leq 2^{\gamma+1}-1$, and $\displaystyle{\frac{x}{\pi}}$ is a nonunit proper divisor of $x$. This implies that 
\[\delta_2(x)\geq 1+N\left(\frac{x}{\pi}\right)+N(x)=1+\frac{N(x)}{N(\pi)}+N(x)\]
\[=1+\frac{(2^{\gamma+1}-1)m}{N(\pi)}+(2^{\gamma+1}-1)m\geq 1+\frac{(2^{\gamma+1}-1)m}{2^{\gamma+1}-1}+(2^{\gamma+1}-1)m\]
\[=1+2^{\gamma+1}m.\] 
However, this contradicts the fact that $\delta_2(x)=2^{\gamma+1}m$, so we conclude that $2^{\gamma+1}-1$ is a prime in $\mathcal O_{\mathbb{Q}(\sqrt{d})}$. Furthermore, because $2^{\gamma+1}-1$ is an integer, we conclude that $2^{\gamma+1}-1$ is an inert integer prime that is also a Mersenne prime. 
\end{proof}
\begin{theorem} \label{Thm2.2}
Let $z$, $m$, $\gamma$, and $x$ be as in Theorem \ref{Thm2.1}. Write $q=2^{\gamma+1}-1$ and $m=q^kv$, where $k\in\mathbb{N}_0$, $v\in\mathbb{N}$, and $q\nmid v$ in $\mathbb{Z}$. Then $k$ is odd, $v\geq q+2$, and
\[m\geq q^{k+1}+(q+3)\sum_{j=0}^{\frac{k-1}{2}}q^{2j}\geq q^2+q+3.\]
\end{theorem}
\begin{proof}
First, note that $q$ is inert and $\upsilon_q(N(x))=k+1$. Therefore, Fact \ref{Fact1.3} implies that $k$ is odd and $\displaystyle{\rho_q(x)=\frac{k+1}{2}}$. Next, assume that $v=1$. Then $m=q^k$, so $x\sim q^{\frac{k+1}{2}}$. This implies that $\displaystyle{\delta_2(x)=\sum_{j=0}^{\frac{k+1}{2}}q^{2j}\equiv 1\imod{q}}$. However, this contradicts Theorem \ref{Thm2.1}, which tells us, under the assumption $m=q^k$, that $\delta_2(x)=2^{\gamma+1}m=(q+1)m=(q+1)q^k\equiv 0\imod{q}$. Therefore, $v>1$. Now, write $\displaystyle{y=\frac{x}{q^{(k+1)/2}}}$. Then, using Theorem \ref{Thm2.1}, 
\[N(y)=\frac{N(x)}{N(q^{\frac{k+1}{2}})}=\frac{qm}{q^{k+1}}=\frac{q^{k+1}v}{q^{k+1}}=v.\]
Because $\displaystyle{\rho_q(x)=\frac{k+1}{2}}$, we see that $y$ and $q^{k+1}$ are relatively prime. Therefore, 
\[\delta_2(x)=\delta_2(y)\delta_2(q^{\frac{k+1}{2}})=\delta_2(y)\sum_{j=0}^{\frac{k+1}{2}}q^{2j}\geq (v+1)\sum_{j=0}^{\frac{k+1}{2}}q^{2j}.\] 
Theorem \ref{Thm2.1} states that $\delta_2(x)=2^{\gamma+1}m=(q+1)m$, so we have 
\[(q+1)m\geq (v+1)\sum_{j=0}^{\frac{k+1}{2}}q^{2j}=q^{k+1}v+q^{k+1}+(v+1)\sum_{j=0}^{\frac{k-1}{2}}q^{2j}\]
\[=qm+q^{k+1}+(v+1)\sum_{j=0}^{\frac{k-1}{2}}q^{2j}.\] 
We can simplify this last inequality to get 
\begin{equation} \label{Eq2.1}
m\geq q^{k+1}+(v+1)\sum_{j=0}^{\frac{k-1}{2}}q^{2j}.
\end{equation} 
Therefore, $\displaystyle{v=\frac{m}{q^k}\geq q+(v+1)\sum_{j=0}^{\frac{k-1}{2}}q^{2j-k}>q}$. As $v$ and $q$ are both odd and $v>q$, we conclude that $v\geq q+2$. Substituting this into \eqref{Eq2.1}, we have 
\[m\geq q^{k+1}+(q+3)\sum_{j=0}^{\frac{k-1}{2}}q^{2j}\geq q^2+q+3,\]
which completes the proof. 
\end{proof}
It is interesting to note that, in the case $z=3+9i$ in $\mathcal O_{\mathbb{Q}(\sqrt{-1})}$, the inequalities in Theorem \ref{Thm2.2} are, in fact, equalities. That is, $q=3$, $v=q+2=5$, and $m=q^2+q+3=15$.  It seems likely, in light of the inequalities in Theorem \ref{Thm2.2}, that the value of $k$ in Theorem \ref{Thm2.2} should have to be $1$.
\par 
We now prove results similar to Theorems \ref{Thm2.1} and \ref{Thm2.2} in the ring $\mathcal O_{\mathbb{Q}(\sqrt{-7})}$. 
\begin{theorem} \label{Thm2.3} 
Let us work in the ring $\mathcal O_{\mathbb{Q}(\sqrt{-7})}$ so that $2$ splits as $2=\varepsilon\overline{\varepsilon}$, where $\varepsilon=\frac{1+\sqrt{-7}}{2}$. Suppose $z$ is $2$-powerfully perfect in $\mathcal O_{\mathbb{Q}(\sqrt{-7})}$ and $2\vert N(z)$ in $\mathbb{Z}$. Then either $z=\varepsilon^\gamma x$ or $z=\overline{\varepsilon}^\gamma x$, where $\gamma\in\mathbb{N}$, $x\in\mathcal O_{\mathbb{Q}(\sqrt{-7})}$, $2\nmid N(x)$ in $\mathbb{Z}$, and $2^{\gamma+1}-1$ is a Mersenne prime that is inert in $\mathcal O_{\mathbb{Q}(\sqrt{-7})}$. Furthermore, there exists an odd positive integer $m$ such that $\delta_2(x)=2^{\gamma+1}m$ and $N(x)=(2^{\gamma+1}-1)m$. 
\end{theorem} 
\begin{proof}
We know that we may write $z=\varepsilon^{\gamma_1}\overline{\varepsilon}^{\gamma_2}x$, where $\gamma_1$, $\gamma_2\in\mathbb{N}_0$, $x\in\mathcal O_{\mathbb{Q}(\sqrt{-7})}$, and $2\!\nmid\! N(x)$ in $\mathbb{Z}$. Furthermore, we know from the fact that $2\vert N(z)$ in $\mathbb{Z}$ that $\gamma_1$ and $\gamma_2$ are not both 
zero. We must prove that either $\gamma_1=0$ or $\gamma_2=0$. Then, after setting $\gamma=\gamma_1+\gamma_2$, we need to prove the final sentence of the theorem and the fact that $2^{\gamma+1}-1$ is a Mersenne prime that is inert in $\mathcal O_{\mathbb{Q}(\sqrt{-7})}$. 
\par  
As $z$ is $2$-powerfully perfect in $\mathcal O_{\mathbb{Q}(\sqrt{-7})}$, we have \[\delta_2(z)=2N(z)=2N(\varepsilon^{\gamma_1} )N(\overline{\varepsilon}^{\gamma_2})N(x)=2^{\gamma_1+\gamma_2+1}N(x).\]
However, we also have 
\[\delta_2(z)=\delta_2(\varepsilon^{\gamma_1})\delta_2(\overline{\varepsilon}^{\gamma_2})\delta_2(x)=\left(\sum_{j=0}^{\gamma_1}
N(\varepsilon^j)\right)\left(\sum_{j=0}^{\gamma_2}N(\overline{\varepsilon}^j)\right)\delta_2(x)\]
\[=\left(\sum_{j=0}^{\gamma_1}2^j\right)\left(\sum_{j=0}^{\gamma_2}2^j\right)\delta_2(x)=(2^{\gamma_1+1}-1)(2^{\gamma_2+1}-1)\delta_2(x).\] 
Therefore, $2^{\gamma_1+\gamma_2+1}
N(x)=(2^{\gamma_1+1}-1)(2^{\gamma_2+1}-1)\delta_2(x)$. As $(2^{\gamma_1+1}-1)(2^{\gamma_2+1}-1)$ is odd, we find that $2^{\gamma_1+\gamma_2+1}\vert\delta_2(x)$ in $\mathbb{Z}$. We may then write $\delta_2(x)=2^{\gamma_1+\gamma_2+1}m$ for 
some positive integer $m$. Substituting this new expression for $\delta_2(x)$ into the equation $2^{\gamma_1+\gamma_2+1} N(x)=(2^{\gamma_1+1}-1)(2^{\gamma_2+1}-1)\delta_2(x)$, we find $N(x)=(2^{\gamma_1+1}-1)(2^{\gamma_2+1}-1)m$. This tells us that $m$ is odd because $2\nmid N(x)$ in $\mathbb{Z}$. Now, $2^{\gamma_1+\gamma_2+1}m=\delta_2(x)\geq 1+N(x)=1+(2^{\gamma_1+1}-1)(2^{\gamma_2+1}-1)m$, so $2^{\gamma_1+\gamma_2+1}>(2^{\gamma_1+1}-1)(2^{\gamma_2+1}-1)=2\cdot2^{\gamma_1+\gamma_2+1}-2^{\gamma_1+1}-2^{\gamma_2+1}+1$. Simplifying this inequality, we have $2^{\gamma_1+1}+2^{\gamma_2+1}>2^{\gamma_1+\gamma_2+1}+1$, which is impossible unless $\gamma_1=0$ or $\gamma_2=0$. Therefore, either $z=\varepsilon^{\gamma_1}x$ or $z=\overline{\varepsilon}^{\gamma_2}x$. Either way, if we write $\gamma=\gamma_1+\gamma_2$, then we have $\delta_2(x)=2^{\gamma+1}m$ and $N(x)=(2^{\gamma+1}-1)m$. Suppose that $2^{\gamma+1}-1$ is not a prime in $\mathcal O_{\mathbb{Q}(\sqrt{-7})}$ so that we may write 
$2^{\gamma+1}-1=y_1y_2$, where $y_1,y_2\in\mathcal O_{\mathbb{Q}(\sqrt{-7})}$ satisfy $1<N(y_1)\leq N(y_2)<N(2^{\gamma+1}-1)=(2^{\gamma+1}-1)^2$. Then, because $N(y_1)N(y_2)=N(2^{\gamma+1}-1)=(2^{\gamma+1}-1)^2$, we see that $N(y_1)\leq 2^{\gamma+1}-1$. Now, let $\pi_0$ be a prime that divides $y_1$. Then $\pi_0\vert N(x)$, which implies that either $\pi_0\vert x$ or $\overline{\pi_0}\vert x$. If $\pi_0\vert x$, write $\pi=\pi_0$. Otherwise, write $\pi=\overline{\pi_0}$.   
Then $N(\pi)\leq N(y_1)\leq 2^{\gamma+1}-1$, and $\displaystyle{\frac{x}{\pi}}$ is a nonunit proper divisor of $x$. This implies that 
\[\delta_2(x)\geq 1+N\left(\frac{x}{\pi}\right)+N(x)=1+\frac{N(x)}{N(\pi)}+N(x)\]
\[=1+\frac{(2^{\gamma+1}-1)m}{N(\pi)}+(2^{\gamma+1}-1)m\geq 1+\frac{(2^{\gamma+1}-1)m}{2^{\gamma+1}-1}+(2^{\gamma+1}-1)m\]
\[=1+2^{\gamma+1}m.\] 
However, this contradicts the fact that $\delta_2(x)=2^{\gamma+1}m$, so we conclude that $2^{\gamma+1}-1$ is a prime in $\mathcal O_{\mathbb{Q}(\sqrt{-7})}$. Furthermore, because $2^{\gamma+1}-1$ is an integer, we conclude that $2^{\gamma+1}-1$ is an inert integer prime that is also a Mersenne prime. 
\end{proof}
\begin{theorem} \label{Thm2.4} 
Let $z$, $m$, $\gamma$, and $x$ be as in Theorem \ref{Thm2.3}. Write $q=2^{\gamma+1}-1$ and $m=q^kv$, where $k\in\mathbb{N}_0$, $v\in\mathbb{N}$, and $q\nmid v$ in $\mathbb{Z}$. Then $k$ is odd, $v\geq q+2$, $\gamma\equiv 1\imod{3}$, $q\equiv 3\imod{7}$, and
\[m\geq q^{k+1}+(q+3)\sum_{j=0}^{\frac{k-1}{2}}q^{2j}\geq q^2+q+3.\]
\end{theorem}
\begin{proof}
Fact \ref{Fact1.4} tells us that an integer prime is inert in $\mathcal O_{\mathbb{Q}(\sqrt{-7})}$ if and only if that integer prime is congruent to $3$, $5$, or $6$ modulo $7$. Also, it is easy to see that powers of $2$ cannot be congruent to $6$ or $7$ modulo $7$. Therefore, as $q$ is a Mersenne prime that is inert in $\mathcal O_{\mathbb{Q}(\sqrt{-7})}$, we must have $q\equiv 3\imod{7}$.  This implies that $2^{\gamma+1}\equiv 4\imod{7}$, so $\gamma\equiv 1\imod{3}$. The proof of the rest of the theorem is identical to the proof of Theorem \ref{Thm2.2}, except all references to Theorem \ref{Thm2.1} should be replaced with references to Theorem \ref{Thm2.3}.    
\end{proof} 
Within the rings $\mathcal O_{\mathbb{Q}(\sqrt{-1})}$, $\mathcal O_{\mathbb{Q}(\sqrt{-2})}$, and $\mathcal O_{\mathbb{Q}(\sqrt{-7})}$, Theorems \ref{Thm2.1} through \ref{Thm2.4} examine some properties of $2$-powerfully perfect numbers with even 
norms. These numbers are somewhat analogous to perfect numbers in $\mathbb{Z}$. The analogues of odd perfect numbers are then $2$-powerfully perfect numbers with odd norms. We now briefly explore some of the properties that such numbers would need to exhibit. 
\begin{theorem} \label{Thm2.5}
Let us work in a ring $\mathcal O_{\mathbb{Q}(\sqrt{d})}$ with $d\in K$. Suppose $z\in\mathcal O_{\mathbb{Q}(\sqrt{d})}\backslash\{0\}$ is such that $I_2(z)=2$ and $N(z)$ is odd (suppose such a $z$ exists). Then 
we may write $z\sim\pi^kx^2$, where  $\pi,x\in\mathcal O_{\mathbb{Q}(\sqrt{d})}\backslash\{0\}$, $\pi$ is prime, and $k\in\mathbb{N}$. Furthermore, $k\equiv N(\pi)\equiv 1\imod{4}$.
\end{theorem}
\begin{proof}
First, let $\pi_0$ be a prime whose norm is odd, and let $\alpha$ be a positive integer. As $\displaystyle{\delta_2(\pi_0^\alpha)=\sum_{j=0}^\alpha N(\pi_0^j)=\sum_{j=0}^\alpha N(\pi_0)^j}$ and $N(\pi_0)$ is odd, we see that $\alpha$ and $\delta_2(\pi_0^\alpha)$ have opposite parities. 
\par 
Now, from $I_2(z)=2$, we have $\delta_2(z)=2N(z)$. 
Because $N(z)$ is odd, we find that $\delta_2(z)\equiv 2\imod{4}$. Write $\displaystyle{z=\prod_{j=1}^r\pi_j^{\alpha_j}}$, where, for all distinct $j,l\in\{1,2,\ldots,r\}$, $\pi_j$ is prime, $\alpha_j$ is a positive integer, and $\pi_j\not\sim \pi_l$. Then $\displaystyle{\delta_2(z)=\prod_{j=1}^r\delta_2(\pi_j^{\alpha_j})}$. Because $\delta_2(z)\equiv 2\imod{4}$, we find that 
there must be exactly one value of $j\in\{1,2,\ldots,r\}$ such that $\delta_2(\pi_j^{\alpha_j})$ is even. This means that there is exactly one value of $j\in\{1,2,\ldots,r\}$ such that $\alpha_j$ is odd. Therefore, $z\sim \pi^kx^2$, where  $\pi,x\in\mathcal O_{\mathbb{Q}(\sqrt{d})}$, $\pi$ is prime, and $k$ is an odd positive 
integer. Furthermore, $\delta_2(\pi^k)\equiv 2\imod{4}$. 
\par 
If $N(\pi)=q^2$, where $q$ is an inert integer prime, then \[\delta_2(\pi^k)=\sum_{l=0}^kN(\pi^l)=\sum_{l=0}^kq^{2l}\equiv\sum_{l=0}^k1\equiv k+1\imod{4}.\] Therefore, in this case, we have 
$k\equiv 1\imod{4}$. Also, because $N(\pi)=q^2$ and $q$ is odd, we know that $N(\pi)\equiv 1\imod{4}$. 
\par 
On the other hand, if $N(\pi)=p$ is an integer prime, then \[\delta_2(\pi^k)=\sum_{l=0}^kN(\pi^l)=\sum_{l=0}^kp^l\equiv 2\imod{4},\] which implies that $p\equiv k\equiv 1\imod{4}$. 
\end{proof} 
\begin{theorem} \label{Thm2.6}
Let us work in a ring $\mathcal O_{\mathbb{Q}(\sqrt{d})}$ with $d\in\{-1,-2\}$. Let $z\in\mathcal O_{\mathbb{Q}(\sqrt{d})}\backslash\{0\}$ be such that $I_2(z)=2$ and $N(z)$ is odd (suppose such a $z$ exists). Then $z$ has at least five nonassociated prime divisors. 
\end{theorem}
\begin{proof}
Suppose $z$ has four or fewer nonassociated prime divisors. Then we may write $z\sim\pi_1^{\alpha_1}\pi_2^{\alpha_2}\pi_3^{\alpha_3}\pi_4^{\alpha_4}$, where, for all distinct $j,l\in\{1,2,3,4\}$, $\pi_j$ is prime, $\alpha_j$ is a nonnegative integer, and $\pi_j\not\sim\pi_l$. 
\par 
First, let us deal with the case $d=-1$. In the ring $\mathcal O_{\mathbb{Q}(\sqrt{-1})}$, the five primes (up to units) that have the smallest odd norms are $2+i$, $1+2i$, $3$, $3+2i$, and $2+3i$, which have norms $5$, $5$, $9$, $13$, and $13$, respectively. Therefore, 
\[I_2(z)=I_2(\pi_1^{\alpha_1}\pi_2^{\alpha_2}\pi_3^{\alpha_3}\pi_4^{\alpha_4})\]
\[=\left(\sum_{j=0}^{\alpha_1}\frac{1}{N(\pi_1)^j}\right)\left(\sum_{j=0}^{\alpha_2}\frac{1}{N(\pi_2)^j}\right)\left(\sum_{j=0}^{\alpha_3}\frac{1}{N(\pi_3)^j}\right)\left(\sum_{j=0}^{\alpha_4}\frac{1}{N(\pi_4)^j}\right)\]
\[<\left(\sum_{j=0}^{\infty}\frac{1}{N(\pi_1)^j}\right)\left(\sum_{j=0}^{\infty}\frac{1}{N(\pi_2)^j}\right)\left(\sum_{j=0}^{\infty}\frac{1}{N(\pi_3)^j}\right)\left(\sum_{j=0}^{\infty}\frac{1}{N(\pi_4)^j}\right)\]
\[\leq\left(\sum_{j=0}^{\infty}\frac{1}{5^j}\right)\left(\sum_{j=0}^{\infty}\frac{1}{5^j}\right)\left(\sum_{j=0}^{\infty}\frac{1}{9^j}\right)\left(\sum_{j=0}^{\infty}\frac{1}{13^j}\right)=\frac{5}{4}\cdot\frac{5}{4}\cdot\frac{9}{8}\cdot\frac{13}{12}<2,\] which is a contradiction. 
\par 
Second, let us deal with the case $d=-2$. In the ring $\mathcal O_{\mathbb{Q}(\sqrt{-2})}$, the integer prime $3$ splits as $3=(1+\sqrt{-2})(1-\sqrt{-2})$. Suppose $1+\sqrt{-2}\vert z$ and $1-\sqrt{-2}\vert z$. Then, because $N(1+\sqrt{-2})=N(1-\sqrt{-2})=3\not\equiv 1\imod{4}$, Theorem \ref{Thm2.5} implies that $1+\sqrt{-2}$ and $1-\sqrt{-2}$ must both appear with even exponents in the prime factorization of $z$. In particular, $(1+\sqrt{-2})^2(1-\sqrt{-2})^2\vert z$. Therefore, by part $(d)$ of Theorem \ref{Thm2.2}, 
\[I_2(z)\geq I_2((1+\sqrt{-2})^2)I_2((1-\sqrt{-2})^2)=\left(1+\frac{1}{3}+\frac{1}{9}\right)^2>2,\]
which is a contradiction. This implies that $1+\sqrt{-2}$ and $1-\sqrt{-2}$ cannot both divide $z$. Now, the six primes (up to units) that have the smallest odd norms are $1+\sqrt{-2}$, $1-\sqrt{-2}$, $3+\sqrt{-2}$, $3-\sqrt{-2}$, $3+2\sqrt{-2}$, and $3-2\sqrt{-2}$, which have norms $3$, $3$, $11$, $11$, $17$, and $17$, respectively. Because $1+\sqrt{-2}$ and $1-\sqrt{-2}$ cannot both divide $z$, we have \[I_2(z)=I_2(\pi_1^{\alpha_1}\pi_2^{\alpha_2}\pi_3^{\alpha_3}\pi_4^{\alpha_4})\]
\[=\left(\sum_{j=0}^{\alpha_1}\frac{1}{N(\pi_1)^j}\right)\left(\sum_{j=0}^{\alpha_2}\frac{1}{N(\pi_2)^j}\right)\left(\sum_{j=0}^{\alpha_3}\frac{1}{N(\pi_3)^j}\right)\left(\sum_{j=0}^{\alpha_4}\frac{1}{N(\pi_4)^j}\right)\]
\[<\left(\sum_{j=0}^{\infty}\frac{1}{N(\pi_1)^j}\right)\left(\sum_{j=0}^{\infty}\frac{1}{N(\pi_2)^j}\right)\left(\sum_{j=0}^{\infty}\frac{1}{N(\pi_3)^j}\right)\left(\sum_{j=0}^{\infty}\frac{1}{N(\pi_4)^j}\right)\]
\[\leq\left(\sum_{j=0}^{\infty}\frac{1}{3^j}\right)\left(\sum_{j=0}^{\infty}\frac{1}{11^j}\right)\left(\sum_{j=0}^{\infty}\frac{1}{11^j}\right)\left(\sum_{j=0}^{\infty}\frac{1}{17^j}\right)=\frac{3}{2}\cdot\frac{11}{10}\cdot\frac{11}{10}\cdot\frac{17}{16}<2,\]
which is a contradiction. 
\end{proof}
\begin{theorem} \label{Thm2.7}
Let us work in the ring $\mathcal O_{\mathbb{Q}(\sqrt{-7})}$. Let $z\in\mathcal O_{\mathbb{Q}(\sqrt{-7})}\backslash\{0\}$ be such that $I_2(z)=2$ and $N(z)$ is odd (suppose such a $z$ exists). Then $z$ has at least eleven nonassociated prime divisors.
\end{theorem}
\begin{proof}
Suppose $z$ has ten or fewer nonassociated prime divisors. Then we may write $\displaystyle{z\sim\prod_{m=1}^{10}\pi_m^{\alpha_m}}$, where, for all distinct $m,l\in\{1,2,\dots,10\}$, $\pi_m$ is prime, $\alpha_m$ is a nonnegative integer, and $\pi_m\not\sim\pi_l$. In $\mathcal O_{\mathbb{Q}(\sqrt{-7})}$, the eleven primes (up to units) that have the smallest odd norms are $\sqrt{-7}$, $3$, $2+\sqrt{-7}$, $2-\sqrt{-7}$, $4+\sqrt{-7}$, $4-\sqrt{-7}$, $5$, $1+2\sqrt{-7}$, $1-2\sqrt{-7}$, $3+2\sqrt{-7}$, and $3-2\sqrt{-7}$, which have norms $7$, $9$, $11$, $11$, $23$, $23$, $25$, $29$, $29$, $37$, and $37$, respectively. Therefore, 
\[I_2(z)=\prod_{m=1}^{10}I_2(\pi_m^{\alpha_m})=\prod_{m=1}^{10}\left(\sum_{j=0}^{\alpha_m}\frac{1}{N(\pi_m)^j}\right)<\prod_{m=1}^{10}\left(\sum_{j=0}^{\infty}\frac{1}{N(\pi_m)^j}\right)\] 
\[\leq \left(\sum_{j=0}^{\infty}\frac{1}{7^j}\right)\left(\sum_{j=0}^{\infty}\frac{1}{9^j}\right)\left(\sum_{j=0}^{\infty}\frac{1}{11^j}\right)\left(\sum_{j=0}^{\infty}\frac{1}{11^j}\right)\left(\sum_{j=0}^{\infty}\frac{1}{23^j}\right)\left(\sum_{j=0}^{\infty}\frac{1}{23^j}\right)\] 
\[\cdot\left(\sum_{j=0}^{\infty}\frac{1}{25^j}\right)\left(\sum_{j=0}^{\infty}\frac{1}{29^j}\right)\left(\sum_{j=0}^{\infty}\frac{1}{29^j}\right)\left(\sum_{j=0}^{\infty}\frac{1}{37^j}\right)\] 
\[=\frac{7}{6}\cdot\frac{9}{8}\cdot\frac{11}{10}\cdot\frac{11}{10}\cdot\frac{23}{22}\cdot\frac{23}{22}\cdot\frac{25}{24}\cdot\frac{29}{28}\cdot\frac{29}{28}\cdot\frac{37}{36}<2,\] 
which is a contradiction.
\end{proof} 
We conclude this section with a remark about $2$-powerfully perfect numbers in $\mathcal O_{\mathbb{Q}(\sqrt{-1})}$, $\mathcal O_{\mathbb{Q}(\sqrt{-2})}$, and $\mathcal O_{\mathbb{Q}(\sqrt{-7})}$ that have odd norms. In each of these three rings, there 
is a prime, say $\xi$, with norm $2$. If $d\in\{-1,-2,-7\}$, $z\in\mathcal O_{\mathbb{Q}(\sqrt{d})}$, $I_2(z)=2$, and $N(z)$ is odd, then $\xi z$ is $2$-powerfully $3$-perfect in $\mathcal O_{\mathbb{Q}(\sqrt{d})}$. This is simply because, under these assumptions, we find that $\displaystyle{I_2(\xi z)=I_2(\xi)I_2(z)=\frac{1+2}{2}I_2(z)=\frac{3}{2}\cdot 2=3}$. 
\section{Further Ideas and a Conjecture} 
We admit that we directed almost all of our attention toward $2$-powerfully perfect numbers, rather than the more general $2$-powerfully $t$-perfect numbers. Hence, the subject of $2$-powerfully $t$-perfect numbers awaits exploration. We also concentrated so heavily on the rings $\mathcal O_{\mathbb{Q}(\sqrt{-1})}$, $\mathcal O_{\mathbb{Q}(\sqrt{-2})}$, and $\mathcal O_{\mathbb{Q}(\sqrt{-7})}$ when dealing with $2$-powerfully perfect numbers that we left open all questions about the rings $\mathcal O_{\mathbb{Q}(\sqrt{d})}$ with $d\!\in\! K$ in which $2$ is inert. We mentioned that $3+9i$  and $9+3i$ are $2$-powerfully perfect and that $30+30i$ is $2$-powerfully $3$-perfect in $\mathcal O_{\mathbb{Q}(\sqrt{-1})}$. Andrew Lelechenko has observed that $84+4788i$ and $1764+4452i$ are also $2$-powerfully $3$-perfect in this ring. Are there other $2$-powerfully $t$-perfect numbers in this ring? What about in other rings? 
\par 
Referring to the concluding paragraph of Section 2, we might ask if there are other relationships between different types of $n$-powerfully $t$-perfect numbers. More specifically, in a given ring $\mathcal O_{\mathbb{Q}(\sqrt{d})}$, are there certain criteria which would guarantee that some specific multiple of an $n_1$-powerfully $t_1$-perfect number is $n_2$-powerfully $t_2$-perfect (for some $n_1,n_2,t_1,t_2\in\mathbb{N}$ with $t_1,t_2\geq 2$)?   
\begin{conjecture} \label{Conj3.1} 
The value of $k$ in Theorem \ref{Thm2.2} must be $1$. Similarly, if there is a $2$-powerfully perfect number in $\mathcal O_{\mathbb{Q}(\sqrt{-7})}$, then the value of $k$ in Theorem \ref{Thm2.4} must be $1$.  
\end{conjecture}

\section{Acknowledgments} 
The author would like to thank Professor Pete Johnson for inviting him to the 2014 REU Program in Algebra and Discrete Mathematics at Auburn University.

\end{document}